\documentclass[a4paper]{article}
 
\usepackage{microtype}

\usepackage{latexsym}
\usepackage{amsmath}
\usepackage{amssymb}
\usepackage{enumerate}
\usepackage{amsthm} 
\usepackage{graphicx} 
\usepackage{amssymb,latexsym,times}
\usepackage{proof}

\newcommand{\tuple}[1]{\vec{#1}}
\newcommand {\indep}[3] {#2 ~\bot_{#1}~ #3}
\newcommand {\indepc}[2] {#1 ~\bot~ #2}

\newcommand{\Dom}{\textrm{Dom}}

\newcommand{\Fr}{\rm Fr}

\newcommand{\Q}{\mathbb{Q}}

\newcommand{\Po}{\mathcal{P}}

\newcommand{\M}{\mathcal{M}}

\newcommand{\on}{\exists}
\newcommand{\ja}{\wedge}
\newcommand{\tai}{\vee}

\renewcommand{\a}{\alpha}
\renewcommand{\b}{\beta}
\newcommand{\g}{\gamma}

\newcommand{\D}{\mathcal{D}}
\def\dep{=\!\!}

\def\dep{=\!\!}
\newcommand{\FO}{{\rm FO}}

\newcommand{\ESO}{{\rm ESO}}
\theoremstyle{plain}
\newtheorem{thm}[equation]{Theorem}
\newtheorem{lem}[equation]{Lemma}
\newtheorem{prop}[equation]{Proposition}
\newtheorem{cor}[equation]{Corollary}

\newtheorem{defi}[equation]{Definition}

\newcommand{\ESOfarity}[1]{{\rm ESO}_f({#1}\mbox{\rm-ary})}

\newcommand{\ESOfvar}[1]{{\rm ESO}_f({#1}\forall)}

\newcommand{\dforall}[1]{\mathcal{D}({#1}\forall)}
\newcommand{\ddep}[1]{\mathcal{D}({#1}\mbox{\rm-dep})}
\newcommand{\RAM}{{\rm RAM}}
\newcommand{\NTIME}{{\rm NTIME}}

\newcommand{\PTIME}{{\rm PTIME}}

\newcommand{\ind}[1]{\FO (\bot_{\rm c})({#1}\mbox{\rm-ind})}
\newcommand{\indNR}[1]{\FO (\bot)({#1}\mbox{\rm-ind})}
\newcommand{\inc}[1]{\FO (\subseteq)({#1}\mbox{\rm-inc})}

\newcommand{\indNRforall}[1]{\FO (\bot) ({#1}\forall)}

\newcommand{\indNRincforall}[1]{\FO (\bot,\subseteq) ({#1}\forall)}
\newcommand{\indlogic}{\FO (\bot_{\rm c})}
\newcommand{\inclogic}{\FO (\subseteq)}
\newcommand{\indNRlogic}{\FO (\bot)}
\newcommand{\indNRinclogic}{\FO (\bot,\subseteq)}

\begin{document}

\author{Pietro Galliani\thanks{Department of Mathematics and Statistics, University of Helsinki, Finland. \texttt{pgallian@gmail.com}} \and Miika Hannula \thanks{Department of Mathematics and Statistics, University of Helsinki, Finland. \texttt{miika.hannula@helsinki.fi}} \and Juha Kontinen\thanks{Department of Mathematics and Statistics, University of Helsinki, Finland. \texttt{juha.kontinen@helsinki.fi}}}
\title{Hierarchies in independence logic\thanks{The authors were supported by grant 264917 of the Academy of Finland.}}

\maketitle

\begin{abstract}
We study the expressive power of  fragments of inclusion and independence logic defined either by restricting the number of universal quantifiers or the arity of inclusion and independence atoms in formulas. Assuming the so-called lax semantics for these logics,  we relate these fragments of inclusion and independence logic to familiar sublogics of existential second-order logic. We also show that, with respect to the stronger strict semantics, inclusion logic is equivalent to existential second-order logic. 
\end{abstract}

\section{Introduction}

Independence logic \cite{gradel10} and inclusion logic \cite{galliani12} are recent variants of dependence logic. Dependence logic \cite{vaananen07}  extends first-order logic by dependence atomic formulas 
\begin{equation}\label{da}\dep(x_1,\ldots,x_n)
\end{equation} the meaning of which is that the value of $x_n$ is completely determined by the values of $x_1,\ldots, x_{n-1}$. 
The semantics of dependence logic is defined using sets of assignments rather than a single assignment as in first-order logic.  
Independence logic replaces the dependence atoms by independence atoms 
$\vec{y}\bot_{\vec{x}} \vec{z}$,
 the intuitive meaning of which is that, with respect to any fixed value of  $ \vec x$, the variables $\vec y$  are totally independent of the variables $\vec z$.
In inclusion logic dependence atoms are replaced by inclusion atoms
$ \vec{x}\subseteq \vec{y},$
meaning that all the values of $\vec{x}$ appear also as values for $\vec{y}$. We study the expressive power of the syntactic fragments of these logics defined either by restricting the number of universal quantifiers or the arity of the independence and inclusion atoms in sentences. These results are proved with respect to lax semantics. We also show that, under strict semantics, inclusion logic is equivalent to existential second-order logic $\ESO$ while, by a recent result of Hella and Galliani \cite{gallhella13}, with lax semantics inclusion logic is equivalent to greatest fixed point logic, and hence to LFP (and $\PTIME$) over finite (ordered) structures.

Since the introduction of dependence logic ($\D$) in 2007 many interesting variants of it have been introduced. In fact the  team semantics of dependence logic has turned  into a general framework for logics in which various notions of dependence and independence can be formalized. Dependence logic has a very intimate and well understood connection to $\ESO$  dating back to the results of \cite{henkin61,enderton70, walkoe70} on Henkin quantifiers. For some of the  new variants and concepts in this area the correspondence to $\ESO$ does not hold. We briefly mention some related work on the  complexity theoretic aspects of these logics:

\begin{itemize}
\item The extension of dependence logic by so-called intuitionistic implication $\rightarrow$ (introduced in \cite{abramsky09}) increases the expressive power of dependence logic to full second-order logic \cite{yang13}. 
\item The model checking problem of full dependence logic, and many of its variants, was recently shown to be NEXPTIME-complete. In fact, for any variant of dependence logic whose atoms are PTIME-computable, the corresponding model checking problem is contained in NEXPTIME  \cite{gradel12}. 
\item The non-classical interpretation of disjunction in dependence logic has the effect that the model checking problems of $\phi_1:=\ \dep(x,y)\vee\dep(u,v)$ and $\phi_2:=\ \dep(x,y)\vee\dep(u,v)\vee \dep(u,v)$ are already  NL-complete and NP-complete, respectively \cite{kontinenj13}. 
\end{itemize}
While dependence logic and independence logic are both equivalent to $\ESO$ in expressive power \cite{vaananen07,gradel10}, for inclusion logic only containment in $\ESO$ has been shown \cite{galliani12}. Furthermore, the expressive power of various natural syntactic fragments of independence and inclusion logics is not understood at the moment. The starting point of our work were the results of \cite{durand11} on the fragments $\dforall{k}$ and  $\ddep{k}$ of dependence logic. The fragment $\dforall{k}$ contains those $\D$-formulas  in which at most $k$ variables have been universally quantified, and in the formulas of $\ddep{k}$  dependence atoms of arity at most $k$ may appear (atoms of the form $\dep(x_1,\ldots,x_n)$ satisfying $n\le k+1$). The following results were shown in  \cite{durand11}: 
\begin{enumerate}
\item \label{arityH}  $\ddep{k}= \ESOfarity{k}$,
\item\label{forallH} $\dforall{k} \le \ESOfvar{k}\le  \dforall{2k}$
\end{enumerate}
where $\ESOfarity{k}$ is  the fragment  of $\ESO$ in which the quantified functions and relations have arity at most $k$, and 
$\ESOfvar{k}$  consists of $\ESO$-sentences that are in Skolem Normal Form and contain at most  $k$ universal first-order quantifiers. The equivalence  in \eqref{arityH} was used to show that in $\ddep{k}$  even cardinality of a $k+1$-ary relation cannot be expressed using the result of Ajtai \cite{ajtai83}. On the other hand,  since 
\[\ESOfvar{k} =\NTIME_{\RAM}(n^k)< \NTIME_{\RAM}(n^{k+1})\]
by \cite{grandjean04} and  \cite{cook72},   an infinite expressivity hierarchy for the fragments  $\dforall{k}$  was shown using  \ref{forallH}. Above $\NTIME_{\RAM}(n^k)$ denotes the family  of classes of $\tau$-structures that can be recognized by a non-deterministic RAM in time $O(n^k)$.  

In \cite{galliani12} it was observed that independence logic and inclusion logic can be given two alternative semantics called strict and lax semantics. For dependence logic these two semantics coincide in the sense that the meaning of any $\D$-formula is the same under both interpretations. For independence and inclusion logic formulas this is not the case as shown in  \cite{galliani12}. In fact, we will show that,  with respect to strict semantics,  inclusion logic is equivalent to $\ESO$, while by a recent result of Hella and Galliani \cite{gallhella13}, with lax semantics inclusion logic is equivalent to greatest fixed point logic. In the rest of the article we consider the expressive power of fragments of independence logic and inclusion logic with lax semantics. First we look at fragments  defined analogously to $\ddep{k}$ of dependence logic. We let $\ind{k}$ contain those independence logic sentences in which independence atoms with at most $k+1$ different variables may appear. Similarly in the sentences of $\inc{k}$ only  inclusion atoms of the form  $\vec{a}\subseteq \vec{b}$, where $|\vec{a}|=|\vec{b}|\leq k$ may appear. Our results show that  
\[ \inc{k}\leq \ESOfarity{k}=\ind{k}. \]
Then we consider the analogoues of $\dforall{k}$ in the case of  $\indNRlogic = \indlogic$ \cite{vaananen13}, which is the  sublogic of independence logic allowing only so-called pure atoms  $\vec{y}\bot \vec{z}$, and  $\indNRinclogic$. We show that 
\begin{itemize}
\item  $\indNRforall{2} =  \indNRlogic$, 
\item  $\indNRincforall{1} = \indNRinclogic$.
\end{itemize}

This article is organized as follows. In Section 2 we  review some basic properties and results regarding dependence logic and its variants. In Section 3 we compare the strict and lax semantics and in Section 4.1  relate the arity  fragments of independence logic and inclusion logic  with that of $\ESO$. Finally,  in Section 4.2 we consider fragments defined by restricting the number of universally quantified variables.









\section{Preliminaries}
\subsection{Team Semantics}
Team semantics is a generalization of Tarski semantics in which formulas are interpreted  by \emph{sets} of assignments, called \emph{teams}, rather than by single assignments. 
In this subsection, we will recall the definition of team semantics for first order logic. We will assume that all our formulas are in negation normal form. Also, all structures considered in the paper are assumed to have at least two elements. 

\begin{defi}

Let $\M$ be a first-order model and $V$  a finite set of variables. Then
\begin{itemize}
\item a \emph{team} $X$ over $\M$ with domain $\Dom(X) = V$ is a finite set of assignments from $V$ to the domain $M$ of $\M$; 

\item for a tuple $\tuple v$ of variables in $V$, we write $X(\tuple v)$ for the set $\{s(\tuple v) : s \in X\}$ of all values that $\tuple v$ takes in $X$, where $s(\tuple v):=( s(v_1),\ldots,s(v_n))$;
\item for a subset $W$ of $V$, we write $X \upharpoonright W$ for the team obtained by restricting all assignments of $X$ to the variables in $W$.
\item For a formula $\phi$, the set of free variables of $\phi$ is denoted by $\Fr(\phi)$.
\end{itemize}
\end{defi}
There exist two variants of team semantics, called respectively \emph{strict} and \emph{lax}, which differ with respect to the interpretation of disjunction and existential quantification. Informally speaking, the choice between strict and lax semantics corresponds to the choice between disallowing or allowing \emph{nondeterministic strategies} in the corresponding semantic games.\footnote{See \cite{galliani12c} and \cite{gradel12} for details.} 

We first give the definition of the lax version of team semantics; later, we will discuss some of the ways in which strict  semantics differs from it. 

\begin{defi}[Team Semantics]
Let $\M$ be any first-order model and let $X$ be any team over it. Then 
\begin{description}
\item[TS-lit:] For all first-order literals $\alpha$, $\M \models_X \alpha$ if and only if, for all $s \in X$, $\M \models_s \alpha$ in the usual Tarski semantics sense; 
\item[TS-$\vee$:] For all $\psi$ and $\theta$, $\M \models_X \psi \vee \theta$ if and only if $X = Y \cup Z$ for two subteams $Y$ and $Z$ such that $\M \models_Y \psi$ and $\M \models_Z \theta$; 
\item[TS-$\wedge$:] For all $\psi$ and $\theta$, $\M \models_X \psi \wedge \theta$ if and only if $\M \models_X \psi$ and $\M \models_X \theta$; 
\item[TS-$\exists$:] For all $\psi$ and all variables $v$, $\M \models_X \exists v \psi$ if and only if there exists a function $H : X \rightarrow \mathcal P(M) \backslash \{\emptyset\}$ such that $\M \models_{X[H/v]} \psi$, where $X[H/v] = \{s[m/v] : s \in X, m \in H(s)\}$; 
\item[TS-$\forall:$] For all $\psi$ and all variables $v$, $\M \models_X \forall v \psi$ if and only if $\M \models_{X[M/v]} \psi$, where $X[M/v] = \{s[m/v] : s \in X, m \in M\}$.
\end{description}
If $\M \models_X \phi$, we say that $X$ \emph{satisfies} $\phi$ in $\M$; and if a sentence (that is, a formula with no free variables) $\phi$ is satisfied by the team $\{\emptyset\}$,\footnote{$\{\emptyset\}$ is the team containing the empty assignment. Of course, this is different from the \emph{empty team} $\emptyset$, containing no assignments.} we say that $\phi$ is \emph{true} in $\M$ and we write $\M \models \phi$.
\end{defi}
In the team semantics setting, formulas $\phi$ and $\psi$ are said to be logically equivalent, $\phi\equiv \psi$, 
if for all models $\M$ and teams $X$, with $\Fr(\phi)\cup \Fr(\psi)\subseteq \Dom(X)$, 
$\M\models_X\phi \Leftrightarrow   \M\models_X\psi$.
Logics $\mathcal{L}$ and  $\mathcal{L}'$ are said to be equivalent, $\mathcal{L}=\mathcal{L}'$, if  every $\mathcal{L}$-sentence $\phi$ is equivalent to some $\mathcal{L}'$-sentence $\psi$, and vice versa. 

The following result can be proved by structural induction on the formula $\phi$: 
\begin{thm}[Flatness]\label{flatness}
For all first order formulas $\phi$ and all suitable models $\M$ and teams $X$, the following are equivalent: 
\begin{enumerate}
\item $\M \models_X \phi$; 
\item For all $s \in X$, $\M \models_{{\{s\}}} \phi$; 
\item For all $s \in X$, $\M \models_s \phi$ according to Tarski semantics. 
\end{enumerate}
\end{thm}

\subsection{Dependencies in Team Semantics}
The advantage of team semantics, and the reason for its development, is that it allows us to extend first-order logic by new atoms 
and operators. 
For the purposes of this paper, the following atoms, inspired by database-theoretic \emph{dependency notions}\footnote{More precisely, dependence atoms correspond to functional dependencies \cite{armstrong74}, independence atoms to embedded multivalued dependencies and conditional dependency conditions as in  \cite{geiger91,naumovre}, and inclusion atoms to inclusion dependencies \cite{fagin81,casanova82}.
}, are of particular interest:
\begin{defi}
\begin{itemize}
\item Let $\tuple x$ be a tuple of variables and let $y$ be another variable. Then $=\!\!(\tuple x, y)$ is a \emph{dependence atom}, with the semantic rule
\begin{description}
\item[TS-dep:] $\M \models_X \dep(\tuple x, y)$ if and only if any two $s, s' \in X$ which assign the same value to $\tuple x$ also assign the same value to $y$;
\end{description}
\item Let $\tuple x$, $\tuple y$, and $\tuple z$ be tuples of variables (not necessarily of the same length). Then $\indep{\tuple x}{\tuple y}{\tuple z}$ is a \emph{conditional independence atom}, with the semantic rule
\begin{description}
\item[TS-ind:] $\M \models_X \indep{\tuple x}{\tuple y}{\tuple z}$ if and only if for any two $s, s' \in X$ which assign the same value to $\tuple x$ there exists a $s'' \in X$ which agrees with $s$ with respect to $\tuple x$ and $\tuple y$ and with $s'$ with respect to $\tuple z$.
\end{description}
Furthermore, we will write $\indepc{\tuple x}{\tuple y}$ as a shorthand for $\indep{\emptyset}{\tuple x}{\tuple y}$, and  call it a \emph{pure independence atom};
\item Let $\tuple x$ and $\tuple y$ be two tuples of variables of the same length. Then $\tuple x \subseteq \tuple y$ 
is an \emph{inclusion atom}, with the semantic rule 
\begin{description}
\item[TS-inc:] $\M \models_X \tuple x \subseteq \tuple y$ if and only if $X(\tuple x) \subseteq X(\tuple y)$; 
\end{description}
\end{itemize}
\end{defi}
Given a collection $\mathcal C \subseteq \{=\!\!(\ldots), \bot_{\rm c}, \subseteq\}$ of atoms, we will write $\FO(\mathcal C)$ (omitting the set parenthesis of $\mathcal{C}$) for the logic obtained by adding them to the language of first-order logic. With this notation dependence logic, independence logic and inclusion logic are denoted by $\FO (\dep(\ldots))$, $\indlogic$ and $\FO(\subseteq)$, respectively. 
We will also write $\indNRlogic$ for the fragment of independence logic containing only pure independence atoms.

All formulas of all the above-mentioned logics satisfy the two following properties:
\begin{prop}[Empty Team Property]
\label{thm:etp}
For all models $\M$ and $\phi \in \FO(=\!\!(\ldots), \bot_{\rm c}, \subseteq)$ over the signature of $\M$, $\M \models_\emptyset \phi$. 
\end{prop}
\begin{prop}[Locality]
\label{thm:loc}
Let $\phi$ be a formula of $\FO(=\!\!(\ldots), \bot_{\rm c}, \subseteq)$ whose free variables $\Fr (\phi)$ are contained in $V$. Then, for all models $\M$ and teams $X$, $\M \models_X \phi$ if and only if $\M \models_{X\upharpoonright V} \phi$.
\end{prop}
Furthermore, we have the two following results for dependence logic: 
\begin{prop}[Downwards Closure]
\label{thm:dc}
For all models $\M$, dependence logic formulas $\phi$ and teams $X$, if $\M \models_X \phi$ then $\M \models_Y \phi$ for all $Y \subseteq X$.
\end{prop}
\begin{thm}[\cite{walkoe70,enderton70,vaananen07}]\label{thm:depESO}
Any dependence logic sentence $\phi$ is logically equivalent to some $\ESO$ sentence $\phi^*$, and vice versa.
\end{thm}

What about independence logic? As shown in \cite{gradel10}, a dependence atom $\dep(\tuple x, y)$ is logically equivalent to the independence atom $\indep{\tuple x}{y}{y}$, and, since independence logic is clearly contained in ESO, we have at once that 
\begin{thm}[\cite{gradel10}]
Any independence logic sentence $\phi$ is logically equivalent to some $\ESO$ sentence $\phi^*$, and vice versa.
\end{thm}
Furthermore,
\begin{thm}[\cite{vaananen13}]
Any independence logic formula is equivalent to some pure independence logic formula.
\end{thm}

For inclusion logic the following is known.
\begin{thm}\label{two} ~
\begin{enumerate}
\item \label{thm:inc_ind}
An inclusion atom $\tuple x \subseteq \tuple y$ is equivalent to the $\indNRlogic$ expression
\[
\forall v_1 v_2 \tuple z ((\tuple z \not = \tuple x \wedge \tuple z \not = \tuple x) \vee (v_1 \not = v_2 \wedge \tuple z \not = \tuple y) \vee ((v_1 = v_2 \vee \tuple z = \tuple y) \wedge \indepc{\tuple z}{v_1 v_2}))
\]
where $v_1$, $v_2$ and $\tuple z$ are new variables \cite{galliani12}.
\item\label{thm:ILGFP} Any inclusion logic sentence $\phi$ is logically equivalent to some \emph{positive greatest fixpoint logic} sentence $\phi^*$, and vice versa \cite{gallhella13}.
\end{enumerate}
\end{thm}


We conclude this subsection with two novel results, a characterization of \emph{dependence} in terms of pure independence and a \emph{prenex normal form theorem} for formulas of our logics.

\begin{thm}\label{thm:DLINDNR}

For all models $\M$ and teams $X$
\[
	\M \models_X =\!\!(\vec{x}, y) \Leftrightarrow \M \models_X \forall \vec{z} \exists w ((\vec{z} = \vec{x} \rightarrow w = y )\wedge \vec{x} y \bot \vec{z} w).
\]
\end{thm}
\begin{proof}
Suppose first that $\M \models_X =\!\!(\vec{x}, y)$. Then there exists a function $f: M^{|\vec{x}|} \rightarrow M$ such that $f(s(\vec{x})) = s(y)$ for all $s \in X$. Then for $Y = X[M/\vec{z}]$, define the choice function $H: Y \rightarrow \Po(M)\setminus \{\emptyset\}$ so that
\[
	H(s) = \{f(s(\vec{z}))\}
\]
for all $s \in Y$, and let $Z = Y[H/w]$. If we can verify that $\M \models_Z \vec{z} = \vec{x} \rightarrow w = y$ and that $\M \models_Z \vec{x} y \bot \vec{z} w$, the left-to-right direction of our proof is done. Now, if $h \in Z$ then $h(w) = f(h(\vec{z}))$ and $h(y) = f(h(\vec{x}))$, and therefore $\M \models_Z \vec{z} = \vec{x} \rightarrow w = y$. Furthermore, for $h, h' \in Z$, we have that $h'' = h[h'(\vec{z})/\vec{z}][h'(w)/w] \in Z$, since our choice of $w$ depends only on $\vec{z}$, and therefore $\M \models_Z \vec{x} y\bot \vec{z} w$. 

Conversely, suppose that there exists a function $H: X[M/\vec{z}] \rightarrow \Po (M)\backslash \{\emptyset\}$ such that, for $Z = X[M/\vec{z}][H/w]$, $\M \models_Z \vec{z} = \vec{x} \rightarrow w = y \wedge \vec{x} y \bot \vec{z} w$. Now let $s, s' \in X$ be such that $s(\vec{x}) = s'(\vec{x}) = \vec{m}$, let $a = s(y)$ and let $b = s'(y)$: we need to prove that $a = b$. 

Take $h \in s[\vec{m}/\vec{z}][H/w] \subseteq Z$: since $\M \models_Z \vec{z} = \vec{x} \rightarrow w = y$, we must have that $h(w) = s(y) = a$. Similarly, for $h' \in s'[\vec{m}/\vec{z}][H/w] \subseteq Z$, we must have that $h'(w) = s'(y) = b$. But $\M \models_Z \vec{x} y \bot \vec{z} w$, so there exists a $h'' \in Z$ such that $h''(\vec{x} y) = h(\vec{x} y) = \vec{m} a$ and $h''(\vec{z} w) = h'(\vec{z} w) = \vec{m} b$. Since, again, $\M \models_Z \vec{z} = \vec{x} \rightarrow w = y$, the only possibility is that $a = b$, as required.
\end{proof}

\begin{lem}\label{aputulos}
Let $\phi, \psi\in \FO(=\!\!(\ldots), \bot_{\rm c}, \subseteq)$ and let $x$ be a variable not occurring free in $\psi$. Then the following equivalences hold:
\begin{enumerate}

\item $\on x \phi \ja \psi \equiv \on x(\phi \ja \psi)$,
\item $\on x \phi \tai \psi \equiv \on x(\phi \tai \psi)$,
\item $\forall x \phi \ja \psi \equiv \forall x(\phi \ja \psi)$,
\item $\forall x \phi \tai \psi \equiv \on a \on b \forall x((\phi\ja a=b) \tai (\psi\ja a \neq b))$ where $a$ and $b$ are new variables.
\end{enumerate}

\end{lem}

\begin{proof}
The cases $1$, $2$ and $3$ are proved as in Lemma $12$ in \cite{hannula13}. We prove number $4$. By Proposition \ref{thm:loc} it is enough to prove the equivalence for teams $X$ with $\Dom(X)=\Fr (\forall x\phi\tai \psi)$. 

Assume first that $\M \models_X \forall x \phi \tai \psi$ and $x$ does not occur free in $\psi$. Then there are $Y\cup Z=X$ such that $\M \models_{Y[M/x]} \phi$ and $\M \models_Z \psi$. Let $0,1\in M$ be distinct. We extend each $s\in X$ with $a\mapsto 0$ and $b\mapsto 0$, for $s \in Y$, and with $a\mapsto 0$ and $b\mapsto 1$, for $s\in Z$, and we let $X'$ consist of these extended assignments. So each $s\in X$ has either one or two extensions in $X'$. Let $Y':=\{s\in X'[M/x] \mid s(a)=s(b)\}$ and $Z' :=\{ s\in X'[M/x]\mid s(a)\neq s(b)\}$. Then by Proposition \ref{thm:loc}, $\M \models_{Y'} \phi \ja a=b$ and $M\models_{Z'} \psi \ja a\neq b$. Hence $M\models_{X'[M/x]} (\phi\ja a=b) \tai (\psi\ja a \neq b)$, and we conclude that $\M \models_X \on a \on b \forall x((\phi\ja a=b) \tai (\psi\ja a \neq b))$.

Assume then that $\M \models_X \on a \on b \forall x((\phi\ja a=b) \tai (\psi\ja a \neq b))$. Let $F_a:X\rightarrow \Po (M)$ and $F_b:X[F_a/a]\rightarrow \Po (M)$ be such that if $X':= X[F_a/a][F_b/b][M/x]$, then $\M\models_{X'} (\phi\ja a=b) \tai (\psi\ja a \neq b)$. Let $Y' \cup Z'=X'$ be such that $\M \models_{Y'} \phi \ja a=b$ and $\M \models_{Z'} \psi \ja a\neq b$. Let $Y:= Y'\upharpoonright \Dom(X)$ and $Z:=Z'\upharpoonright \Dom (X)$. Then $Y[M/x]=Y' \upharpoonright (\Dom (X) \cup \{x\})$, and thus by Proposition \ref{thm:loc} $\M \models_{Y[M/x]} \phi$. Also by Proposition \ref{thm:loc} $\M \models_Z \psi$. Since $Y\cup Z = X$, we conclude that $\M \models \forall x \phi \tai \psi$.
\end{proof}

Lemma \ref{aputulos} allows us to show the following.
\begin{thm}\label{aputulos2}
Any formula $\phi \in \FO(=\!\!(\ldots), \bot_{\rm c}, \subseteq)$ is logically equivalent to some formula $\phi'$ such that 
\begin{enumerate}
\item $\phi'$ is of the form $Q_1 x_1 \ldots Q_k x_k \psi$, where $\psi$ is quantifier-free; 
\item Any literal or non-first-order atom which occurs in $\phi'$ occurred already in $\phi$; 
\item The number of universal quantifiers in $\phi'$ is the same as the number of universal quantifiers in $\phi$.
\end{enumerate}
\end{thm}

\section{Comparing strict and lax semantics}
As we mentioned, there exists an alternative variant of lax semantics, called strict semantics. It differs from lax semantics in the definition of the semantic rules for disjunction and existential quantification, which are replaced respectively by 
\begin{description}
\item[STS-$\vee$:] For all $\psi$ and $\theta$, $\M \models_X \psi \vee \theta$ if and only if $Y$ and $Z$ exist such that $Y \cup Z=X$, $Y\cap Z = \emptyset$, $\M \models_Y \psi$ and $\M \models_Z \theta$; 
\item[STS-$\exists$:] For all $\psi$ and all variables $v$, $\M \models_X \exists v \psi$ if and only if there exists a function $F : X \rightarrow M$ such that $\M \models_{X[F/v]} \psi$, where $X[F/v] = \{s[F(s)/v] : s \in X\}$.
\end{description}
It is clear that 
\begin{prop}
If $\M \models_X \phi$ according to strict team semantics, then $\M \models_X \phi$ according to lax team semantics.
\end{prop}
For downwards closed logics, such as dependence logic, the converse is also true. 
\begin{prop}[\cite{galliani12}]
For all dependence logic formulas $\phi$, models $\M$ and teams $X$, $\M \models_X \phi$ holds wrt strict team semantics if and only if it holds wrt lax team semantics.
\end{prop}
However, the same is false for both  inclusion logic and independence logic. In particular, as we will now see, inclusion logic with strict semantics is equivalent to full existential second order logic, in contrast with the second item of Theorem \ref{two}.

By Theorem \ref{thm:depESO}, it suffices to show that every dependence logic sentence is equivalent to some inclusion logic sentence (with strict semantics). In order to do so, we will use the following \emph{normal form theorem} from \cite{vaananen07}: 
\begin{thm}[\cite{vaananen07}]
Every dependence logic sentence is equivalent to some sentence of the form 
\begin{equation}
	\phi := \forall \vec{x} \exists \vec{y} \left(\bigwedge_{y_i \in \vec{y}} =\!\!(\vec{v_i}, \vec{y_i}) \wedge \theta\right)
	\label{eq:eq1}
\end{equation}
where for all $i$, $\vec{v_i}$ is contained in $\vec{x}$ and where $\theta$ is a quantifier-free first-order formula. 
\end{thm}
As we will now show, in strict semantics the dependence atoms in \eqref{eq:eq1} can be replaced by equivalent inclusion logic subformulas; and, therefore, it follows at once that (strict) inclusion logic is equivalent to dependence logic (and, therefore, to ESO) over sentences. 

\begin{defi}
Let $\M$ be a model and $X$  a team, and let $\vec{x}$ be a tuple of variables in its domain. We say that $X$ is $\vec{x}$-\emph{universal} if for all tuples of elements $\vec{m}$ with $|\vec{m}| = |\vec{x}|$, there exists one and only one $s \in X$ with $s(\vec{x}) = \vec{m}$. 
\end{defi}
\begin{lem}\label{that}
If $X$ is of the form $\{\emptyset\}[M/\vec{x}][\vec{F}/\vec{y}]$ then $X$ is $\vec{x}$-universal.
\end{lem}
\begin{proof}
Obvious (but note that if the $\vec{F}$ were replaced by \emph{nondeterministic} choice functions $\vec{H}$, as in the case of the lax semantics, this would not hold).
\end{proof}

\begin{prop}\label{this}
Let $\M$ be a model and $X$  a $\vec{x}$-universal team.  Suppose also that $y \not \in \vec{x}$,  $\vec{v} \subseteq \vec{x}$, and  $\vec{w} = \vec{x} \backslash \vec{v}$ (that is, $\vec{w}$ lists, without repetitions, all variables occurring in $\vec{x}$ but not in $\vec{v}$). Then
\[\M \models_X =\!\!(\vec{v}, y) \Leftrightarrow \M \models_X \forall \vec{q} (\vec{q} \vec{v} y \subseteq \vec{w} \vec{v} y).\]

\end{prop}
\begin{proof}
Suppose that $\M \models_X =\!\!(\vec{v}, y)$, and let $h = s[\vec{m} / \vec{q}] \in X[M/\vec{q}]$, where $s \in X$. Since $X$ is $\vec{x}$-universal and $\vec{x} = \vec{v} \cup \vec{w}$, there exists an assignment $s' \in X$ such that $s'(\vec{w}) = \vec{m}$ and $s'(\vec{v}) = s(\vec{v})$. Since $y$ is a function of $\vec{v}$ alone, this implies that $s'(y) = s(y)$. Finally, $h' = s'[\vec{m}/\vec{q}] \in X[M/\vec{q}]$, and $h'(\vec{w} \vec{v} y) = \vec{m} s(\vec{v} y) = h(\vec{q} \vec{v} y)$, as required. 

Conversely, suppose that $\M \models_X \forall \vec{q} (\vec{q} \vec{v} y \subseteq \vec{w} \vec{v} y)$, and let $s, s' \in X$ be such that $s(\vec{v}) = s'(\vec{v})$. Now let $\vec{m} = s'(\vec{w})$, and consider $h = s[\vec{m}/\vec{q}] \in X[M/\vec{q}]$. By hypothesis, there exists a $h' \in X[M/\vec{q}]$ such that $h'(\vec{w}) = h(\vec{q}) = \vec{m}$ and $h'(\vec{v} y) = h(\vec{v} y) = s(\vec{v} y)$. This $h'$ is of the form $s''[\vec{m'} / \vec{q}]$ for some $s'' \in X$; and for this $s''$, we have that $s''(\vec{v}) = s(\vec{v}) = s'(\vec{v})$, $s''(\vec{w}) = \vec{m} = s'(\vec{w})$ and $s''(y) = s(\vec{y})$. Now, $\vec{x} = \vec{v} \cup \vec{w}$, and $s''$ coincides with $s'$ over it, and $X$ is $\vec{x}$-universal; therefore, we have to conclude that $s'' = s'$. But then $s'(y) = s''(y) = s(y)$, and therefore $s''$ and $s$ coincide over $y$ too. 
\end{proof}

\begin{cor}
With strict semantics inclusion logic is equivalent to $\ESO$.
\end{cor}
\begin{proof}
By Lemma \ref{that} and the Proposition \ref{this}, any sentence of the form \eqref{eq:eq1} can be expressed in inclusion logic as
\begin{equation}
	\forall \vec{x} \exists \vec{y} \left(\bigwedge_{y_i \in \vec{y}} (\forall \vec{q}_i (\vec{q}_i \vec{v}_i y \subseteq \vec{w}_i \vec{v}_i y)) \wedge \theta\right)
\end{equation}
where for all $i$, $\vec{w}_i = \vec{x} \backslash \vec{v}_i$; and this implies our result.
\end{proof}

The analogue of Theorem \ref{thm:loc} (locality) for inclusion logic with strict semantics fails. As an especially surprising example of such an failure we now show that one can find inclusion logic sentences that count the number of assignments in a team: 
\begin{thm}\label{outo}
For each natural number $n$ there is a sentence $\phi \in \inclogic$ such that for all models $\M$ and teams $X$ where $X \neq \emptyset$ and the variables in $\Dom (X)$ do not appear in $\phi$,
\[\label{equiv}\M \models_X \phi \textrm{ if and only if } |X| \geq n.\]

\end{thm}

\begin{proof}
Let $n$ be a natural number. We may assume that $n \geq 2$ because in the case $n=1$ we can just choose $\phi:= \top$. Let $\vec{x}_i$, for $0 \leq i \leq n-1$, list variables $x_{i,0}, \ldots ,x_{i,l}$ where $l=$log$(n)$. Let
$$\phi:= \on \vec{x}_0 \ldots \on \vec{x}_{n-1} (\bigwedge_{0 \leq i \leq n-1} \vec{x}_i \subseteq \vec{x}_0 \ja \bigwedge_{0 \leq i < j \leq n-1} \vec{x}_i \neq   \vec{x}_j)$$
where
$$\vec{x}_i \neq \vec{x}_j := \bigvee_{0 \leq k \leq l} x_{i,k} \neq x_{j,k}.$$
Now $\phi$ is as wanted:

Assume first that $\M \models_X \phi $. Then there are, for $0 \leq i \leq n-1$, functions
$$F_i: X[F_{0}/\vec{x}_{0}]\ldots[F_{i-1}/\vec{x}_{i-1}]  \rightarrow M^{l+1}$$
such that
\begin{equation}\label{strictex}
M \models_{X'} \bigwedge_{0 \leq i \leq n-1} \vec{x}_i \subseteq \vec{x}_0 \ja \bigwedge_{0 \leq i < j \leq n-1} \vec{x}_i \neq   \vec{x}_j
\end{equation}
when $X':= X[F_0/\vec{x}_0]\ldots[F_{n-1}/\vec{x}_{n-1}]$. Let $s \in X'$ be some arbitrary assignment.  From \eqref{strictex} it follows that $X'$ must include assignments $s_i$, for $0 \leq i \leq n-1$, such that $s_i(\vec{x}_0) = s(\vec{x}_i)$. Also from \eqref{strictex} it follows that $s(\vec{x}_i) \neq s(\vec{x}_j)$, for $0 \leq i<j \leq n-1$. Thus the assignments $s_i$ are distinct and therefore $|X'| \geq n$. Because existential quantification of new variables in strict semantics preserves the cardinality of a team we deduce that $X \geq n$.

Suppose then $X \geq n$. By the assumption $n \geq 2$, and thus we may deduce that $|M| \geq 2$. Let $0$ and $1$ be two different members of $M$, and let $\overline{i}$ be the binary representation (of length $l+1$) of $i$, for $0 \leq i \leq n-1$, in terms of these $0$ and $1$. Choose then $n$ different assignments $s_0, \ldots ,s_{n-1}$ from $X$. We define, for $0 \leq i \leq n-1$, $F_i:  X[F_0/\vec{x}_0]\ldots[F_{i-1}/\vec{x}_{i-1}] \rightarrow M^{l+1}$ as follows:
\[\begin{array}{ll}
$$F_i(s):=  \left\{\begin{array}{l l}
   \overline{j+i} &\textrm{ if }s \upharpoonright \Dom (X) = s_j,\textrm{ for }0 \leq j \leq n-1,\\
    \overline{i}& \quad \textrm{otherwise}\\
  \end{array}\right.$$
\end{array}\]
where $j+i$ is mod $n$. By the assumption, the variables in $\Dom(X)$ are not listed in $\vec{x}_0\ldots \vec{x}_{n-1}$, and thus the functions $F_i$ are consistent with the definition of existential quantification for strict semantics. Without the assumption it could be the case that different $s_i$ and $s_j$ would collapse into one assignment in the quantification procedure. Let $X':= X[F_0/\vec{x}_0]\ldots[F_{n-1}/\vec{x}_{n-1}]$. Then $s_j$, for $0 \leq j \leq n-1$, is extended in $X'$ to 
$$s_j(\overline{j}/\vec{x}_{\vec{0}})(\overline{j+1}/\vec{x}_1)\ldots (\overline{j-2}/\vec{x}_{n-2})(\overline{j-1}/\vec{x}_{n-1}),$$
and each $t\in X\setminus \{s_j \mid 0 \leq j \leq n-1\}$ is extended in $X'$ analogously to $s_0$. So for each $s \in X'$ and $0 \leq i < j \leq n-1$ it holds that $s(\vec{x}_i) \neq s(\vec{x}_j)$. Also
\[\{s({\vec{x}_0) \mid s \in X'}\} = \{\overline{i} \mid 0 \leq i \leq n-1\} = \bigcup_{0 \leq i \leq n-1} \{s(\vec{x}_i) \mid s \in X'\},\]
and thus
\[M \models_{X'}  \bigwedge_{0\leq i \leq n-1} \vec{x}_i \subseteq \vec{x}_0 \ja \bigwedge_{0 \leq i < j \leq n-1} \vec{x}_i \neq   \vec{x}_j\]
which concludes the proof. 
\end{proof}

The failure of locality in non-downwards closed logics with strict semantics is somewhat problematic, as it causes the interpretation of a formula to depend on the values that our assignments take on variables which do \emph{not} occur in it. As a consequence, in the rest of this work we will focus on logics with lax semantics. 
\section{The expressive power of fragments}

The purpose of this section is to generalize the classification of the expressive power of fragments of dependence logic of \cite{durand11} to the case of other variants (with respect to lax semantics). We will consider the following fragments.
\begin{defi} Let $\mathcal{C}$ be a subset of $\{\dep(\ldots),\bot_{\rm c},\bot,\subseteq\}$ and let $k \in \mathbb{N}$. Then
\begin{enumerate}
\item $\FO(\mathcal{C}) (k-$dep$)$ is the class of sentences of $\FO(\mathcal{C})$ in which dependence atoms of the form $\dep(\vec{z},y)$, where $\vec{z}$ is of length at most $k$, may appear.
\item $\FO(\mathcal{C}) (k-$ind$)$ is the class of sentences of $\FO(\mathcal{C})$ in which independence atoms of the form $\vec{y}\bot_{\vec{x}} \vec{z}$, where  $\vec{x}\vec{y}\vec{z}$  has at most $k+1$ distinct variables, may appear.
\item $\FO(\mathcal{C}) (k-$inc$)$ is the class of sentences of $\FO(\mathcal{C})$ in which inclusion atoms of the form $\vec{a} \subseteq \vec{b}$, where $\vec{a}$ and $\vec{b}$ are of length at most $k$, may appear.
\item $\FO(\mathcal{C}) (k\forall)$ is the class of sentences of $\FO(\mathcal{C})$ in which at most $k$ universal quantifiers occur. 
\end{enumerate}
\end{defi}
As in \cite{durand11}, we will write $\ddep{k}$ and $\dforall{k}$ for $\FO(\dep(\ldots )) (k-$dep$)$ and $\FO(\dep(\ldots ))(k\forall)$, respectively.

\subsection{Arity hierarchies}
In this section we will prove that $\ind{k}=\ESOfarity{k}$. In particular this also implies that $\ind{k}=\ddep{k}$ \cite{durand11}. We will also prove that $\inc{k}\leq \ESOfarity{k}$. The direction from $\ESOfarity{k}$ to $\ind{k}$ is straightforward.

\begin{prop}\label{leq}
$\ESOfarity{k}\leq \ind{k}.$
\end{prop}
\begin{proof}
Let $\phi \in\ESOfarity{k}$. By \cite{durand11} there exists a $\phi' \in\ddep{k}$ equivalent to $\phi$ and of the form
$$Q^1 x_1\ldots Q^m x_m \on y_1 \ldots \on y_n (\bigwedge_{1 \leq j \leq n} \dep(\vec{z}_j,y_j) \ja \theta)$$
where $\vec{z}_j$, for $1 \leq j \leq n$, is a sequence of length at most $k$. By \cite{gradel10} each dependence atom $\dep(\vec{z},y)$ is equivalent to the independence atom $\indep{\vec{z}}{y}{y}$. Therefore we can present $\phi'$ in the following independence logic form
$$Q^1 x_1\ldots Q^m x_m \on y_1 \ldots \on y_n (\bigwedge_{1 \leq j \leq n}  y_j\bot_{\vec{z}_j} y_j \ja \theta)$$
where $\vec{z}_j y_j$, for $1 \leq j \leq n$, is a sequence of at most $k+1$ different variables.
\end{proof}
We will next show the other direction.

\begin{lem}\label{disjoint}
Let $\vec{b} \bot_{\vec{a}} \vec{c}$ be an independence atom where $\vec{a}$, $\vec{b}$ and $\vec{c}$ are tuples of variables. If $\vec{b_0}$ lists the variables in $\vec{b} - \vec{a}\cup\vec{c}$, $\vec{c_0}$ lists the variables in $\vec{c}-\vec{a}\cup\vec{b}$, and $\vec{d}$ lists the variables in $\vec{b} \cap \vec{c}-\vec{a}$, then
$$\vec{b} \bot_{\vec{a}} \vec{c} \equiv \vec{b_0}\bot_{\vec{a}}\vec{c_0} \ja \bigwedge_{d\in \vec{d}} \dep(\vec{a},d).$$
\end{lem}

\begin{proof}
Assume that $\M \models_X \vec{b} \bot_{\vec{a}} \vec{c}$. Then clearly $\vec{b_0}\bot_{\vec{a}}\vec{c_0}$. For $\bigwedge_{d\in \vec{d}} \dep(\vec{a},d)$, let $d \in \vec{d}$ and $s,s'\in X$ be such that $s(\vec{a})=s'(\vec{a})$. Then by the assumption there is $s'' \in X$ such that $s''(\vec{a}\vec{b}\vec{c})=s(\vec{a}\vec{b})s'(\vec{c})$. Because $d$ is listed in both $\vec{b}$ and $\vec{c}$, it follows that $s(d)=s'(d)$.

Suppose then $\M \models_X \vec{b_0}\bot_{\vec{a}}\vec{c_0} \ja \bigwedge_{d\in \vec{d}} \dep(\vec{a},d)$. Let $s,s' \in X$ be such that $s(\vec{a})=s'(\vec{a})$. By the assumption there is $s''\in X$ such that $s''(\vec{a}\vec{b}_0\vec{c}_0)=s(\vec{a}\vec{b}_0)s'(\vec{c}_0)$. We want to show that $s''(\vec{a}\vec{b}\vec{c})=s(\vec{a}\vec{b})s'(\vec{c})$. Consider first variables $x$ listed in $\vec{b} - \vec{b}_0$. If $x$ is listed in $\vec{a}$, then $s''(x)=s(x)$ as wanted. Assume that $x$ is listed in $\vec{c}-\vec{a}$. Then $x \in \vec{d}$, and thus $s''(x)=s(x)$ follows from $s''(\vec{a})=s(\vec{a})$.

For variables $x$ is listed in $\vec{c} - \vec{c}_0$ the proof of $s''(x)=s'(x)$ is analogous because $s(\vec{a})=s'(\vec{a})$. This concludes the proof.
\end{proof}

Now we can prove the following proposition. In the proof we will present a translation from independence logic to $\ESO$, where independence atoms are coded by relation variables preserving the arity of the atoms. Note that the translation presented in \cite{gradel10} does not preserve this property.
\begin{prop}\label{geq}
$\ind{k}\leq \ESOfarity{k}.$
\end{prop}
\begin{proof}
Let $\phi \in \ind{k}$. By Theorem \ref{aputulos2} we may assume that $\phi$ is in prenex normal form $Q^1 x_1 \ldots Q^n x_n \theta$ where $\theta$ is a quantifier-free formula. By Lemma \ref{disjoint} we may assume that each independence atom in $\theta$ is either of the form $\dep(\vec{z},y)$ or $\vec{b} \bot_{\vec{a}} \vec{c}$ where 
\begin{itemize}
\item $y$ is not listed in $\vec{z}$,
\item $\vec{a}$, $\vec{b}$ and $\vec{c}$ do not share any  variables,
\item $|\vec{z}|\leq k$ and $|\vec{a}\vec{b}\vec{c}| \le k+1$. 
\end{itemize}

Let us next consider the subformulas of $\theta$. We will enumerate the subformulas of $\theta$ by $\theta_{\vec{i}}$ where $\vec{i}$ is a binary sequence encoding the location of the subformula in $\theta$. Let $\theta_{\lambda}:=\theta$ where $\lambda$ is the empty sequence. If
 $\theta_{\vec{i}}$ is a conjunction (or a disjunction), then we denote its conjuncts (or the disjuncts) as $\theta_{\vec{i}0}$ and $\theta_{\vec{i}1}$.
Now let $S:=\{\vec{i}\mid \theta_{\vec{i}}\textrm{ is a subformula of } \theta\}$, and let $D$ and $I$ be the subsets of $S$ consisting of sequences $\vec{i}$ for which $\theta_{\vec{i}}$ is a dependence atom or an independence atom, respectively. Let $\leq$ be a partial order in $S$ where $\vec{i} \leq \vec{j}$ if $\vec{i}\vec{k}=\vec{j}$ for some binary $\vec{k}$. Then $\vec{i} \leq \vec{j}$ if and only if $\theta_{\vec{j}}$ is a subformula of $\theta_{\vec{i}}$.

Next we will define a $\Phi \in \ESOfarity{k}$ equivalent to $\phi$. First we define $\varphi_{\vec{i}}$ for each $\vec{i} \in S$ inductively as follows:
\begin{itemize}
\item $\varphi_{\vec{i}} := \theta_{\vec{i}}$ if $\theta_{\vec{i}}$ is a first-order atom,
\item $\varphi_{\vec{i}} := S_{\vec{i}}(\vec{a}\vec{b}) \ja T_{\vec{i}}(\vec{a}\vec{c})$ if $\theta_{\vec{i}}$ is $\indep{\vec{a}}{\vec{b}}{\vec{c}}$,
\item $\varphi_{\vec{i}} := f_{\vec{i}}(\vec{z})=y$ if $\theta_{\vec{i}} $ is $ \dep(\vec{z},y)$,
\item $\varphi_{\vec{i}} := \varphi_{\vec{i}0} \ja \varphi_{\vec{i}1}$ if $\theta_{\vec{i}}$ is $\theta_{\vec{i}0} \ja \theta_{\vec{i}1}$,
\item $\varphi_{\vec{i}} := \varphi_{\vec{i}0} \tai \varphi_{\vec{i}1}$ if $\theta_{\vec{i}}$ is $\theta_{\vec{i}0} \tai \theta_{\vec{i}1}$.
\end{itemize}
Now let $\varphi:= \varphi_{\lambda}$. Then $\varphi$ is a quantifier-free first-order formula sharing the structure of $\theta$ where the dependence and independence atoms are interpreted using new function symbols $f_{\vec{i}}$ and  relation symbols $S_{\vec{i}}$ and $T_{\vec{i}}$, respectively. Let $\vec{z}_{\vec{i}}$, for $\vec{i} \in I$, list the variables in $\{x_1, \ldots ,x_n\} \setminus \Fr (\theta_{\vec{i}})$. In the following, for example, $\on (S_{\vec{i}})_{\vec{i} \in I} $ denotes the prefix $\on S_{\vec{i}_1} \ldots \on S_{\vec{i}_m}$ where $\vec{i}_1, \ldots , \vec{i}_m$ enumerates $I$. So let us define $\Phi$ as

\begin{align}\label{Phi}
\on (S_{\vec{i}})_{\vec{i}\in I} (T_{\vec{i}})_{\vec{i}\in I}  (f_{\vec{i}})_{\vec{i}\in D} (Q^1 x_1 \ldots Q^n x_n \varphi \ja \Omega)
\end{align}
where
\begin{align}\label{Omega}
\Omega := \bigwedge_{\vec{i}\in I} [\forall \vec{a}\vec{b}\vec{c} (S_{\vec{i}}(\vec{a}\vec{b}) \ja T_{\vec{i}}(\vec{a}\vec{c}) )\rightarrow \on \vec{z}_{\vec{i}} (\bigwedge_{\vec{j} \leq \vec{i}}\varphi_{\vec{j}} \ja Q^1 x'_1 \ldots Q^n x'_n (\varphi' \ja \chi))]
\end{align}
where $\varphi':=\varphi(x'_1/x_1)\ldots(x'_n/x_n)$ and
\begin{align}\label{chi}
\chi:= \bigwedge_{\substack{1\leq k\leq n\\Q^k=\on}}  (x_1=x'_1 \ja \ldots \ja x_{k-1}=x'_{k-1}) \rightarrow x_{k} =x'_{k}.
\end{align}

The idea behind $\Phi$ is that the relation variables $S_{\vec{i}}$ and $T_{\vec{i}}$, for $\vec{i} \in I$, encode a subteam $X_{\vec{i}}$ that satisfies $\indep{\vec{a}}{\vec{b}}{\vec{c}}$. Then $\Omega$ will ensure that for each $s,s' \in X_{\vec{i}}$ with $s(\vec{a})=s'(\vec{a})$ there is $s''$ corresponding to the values of $\vec{a} \vec{b}\vec{c}$ and $\vec{z}_{\vec{i}}$ such that $s''(\vec{a}\vec{b}\vec{c})=s(\vec{a}\vec{b})s'(\vec{b})$. The variables $x_i'$ and $\chi$ will ensure that $s''\in X_{\vec{i}}$.
We will now prove that
$$\M \models \phi \Leftrightarrow \M \models \Phi.$$

\textit{Only if-part:} Assume that $\M \models \phi$. Then there are functions
$$F_i:X[F_1/x_1] \ldots [F_{i-1}/x_{i-1}] \rightarrow \Po (M) \setminus \{\emptyset\},$$
for $1\leq i \leq n$, such that
$$\M \models_Y  \theta$$
when $Y:=\{\emptyset\}[F_1/x_1] \ldots [F_n/x_n]$. Note that $F_i(s)= M$ if $Q^i = \forall$.

Let us then construct teams $Y_{\vec{i}}$, for $\vec{i} \in S$, such that $\M  \models_{Y_{\vec{i}}} \theta_{\vec{i}}$, as follows. Let $Y_{\lambda}:= Y$.
\begin{itemize}
\item Assume that $\M \models_{Y_{\vec{i}}}  \theta_{\vec{i}}$ where $\theta_{\vec{i}} = \theta_{\vec{i}0} \ja \theta_{\vec{i}1}$. Then $Y_{\vec{i}0}:= Y_{\vec{i}}$ and $Y_{\vec{i}1}:=Y_{\vec{i}}$.

\item Assume that $\M \models_{Y_{\vec{i}}} \theta_{\vec{i}}$ where $\theta_{\vec{i}} = \theta_{\vec{i}0} \tai \theta_{\vec{i}1}$. Then choose $Y_{\vec{i}0} \cup Y_{\vec{i}1} = Y_{\vec{i}}$ so that $\M \models_{Y_{\vec{i}0}}  \theta_{\vec{i}0}$ and $\M \models_{Y_{\vec{i}1}} \theta_{\vec{i}1}$.
\end{itemize}

We then note that 
\begin{align}
\label{indep}
\M \models_{Y_{\vec{i}}}\indep{\vec{a}}{\vec{b}}{\vec{c}}\hspace{0,5cm} &\textrm{    if    } \hspace{0,5cm} \theta_{\vec{i}}\textrm{ is }\indep{\vec{a}}{\vec{b}}{\vec{c}},\\
\label{dep}
\M \models_{Y_{\vec{i}}}  \dep(\vec{z},y) \hspace{0,5cm}&\textrm{     if    }\hspace{0,5cm} \theta_{\vec{i}}\textrm{ is }\dep(\vec{z},y).
\end{align}
Now, for $\theta_{\vec{i}}$ of the form $\indep{\vec{a}}{\vec{b}}{\vec{c}}$, the interpretations of $S_{\vec{i}}$ and $T_{\vec{i}}$ will be the following:
\begin{align*}
&S_{\vec{i}}^{\M} := \{s(\vec{a}\vec{b}) \mid s\in Y_{\vec{i}}\},\\
&T_{\vec{i}}^{\M} := \{s(\vec{a}\vec{c}) \mid s\in Y_{\vec{i}}\}.
\end{align*}
For $\theta_{\vec{i}}$ of the form $\dep(\vec{z},y)$ we interpret $f_{\vec{i}}$ as follows:
\[\begin{array}{ll}
$$f^{\M}_{\vec{i}}(\vec{a}):=  \left\{\begin{array}{l l}
   b &\textrm{ if } s(\vec{z}y)=\vec{a}b \textrm{ for some } s\in Y_{\vec{i}},\\
    0 & \quad \textrm{otherwise}\\
  \end{array}\right.$$
\end{array}\]
where $0 \in M$ is arbitrary. Now $f_{\vec{i}}$ is well defined by \eqref{dep}. Let then $\M^* := (\M, \vec{S}^{\M},\vec{T}^{\M},\vec{f}^{\M})$. We will show that
$$\M^* \models Q^1 x_1 \ldots Q^n x_n \varphi \ja \Omega.$$
Consider the first conjunct. For each $x_i$ with $Q^i = \on$ we can choose a value for it so that the values of $x_1, \ldots ,x_i$ agree with some $s \in Y$. Thus it suffices to show that, for $s \in Y$,
$$\M^* \models_s \varphi.$$
Since $\varphi$ is a first-order formula, by Theorem \ref{flatness} it suffices to show that $\M^* \models_Y \varphi$. This can be done inductively: For each atomic $\varphi_{\vec{i}}$, $\M^* \models_{Y_{\vec{i}}} \varphi_{\vec{i}}$ by the definitions. If $\M^* \models_{Y_{\vec{i}0}}\varphi_{\vec{i}0}$ and $\M^* \models_{Y_{\vec{i}1}} \varphi_{\vec{i}1}$, and $\varphi_{\vec{i}}$ is either disjunction or conjunction of $\varphi_{\vec{i}0}$ and $\varphi_{\vec{i}1}$, then $\M^* \models_{Y_{\vec{i}}} \varphi_{\vec{i}}$ by the construction of $Y_{\vec{i}}$. This concludes the claim and thus the first conjunct part.

Next we will to show that $\M^* \models \Omega$ where $\Omega$ is the formula
$$\bigwedge_{\vec{i}\in I} [\forall \vec{a}\vec{b}\vec{c} (S_{\vec{i}}(\vec{a}\vec{b}) \ja T_{\vec{i}}(\vec{a}\vec{c}) )\rightarrow \on \vec{z}_{\vec{i}} (\bigwedge_{\vec{j} \leq \vec{i}} \varphi_{\vec{j}} \ja Q^1 x'_1 \ldots Q^n x'_n (\varphi' \ja \chi))].$$
Let $\vec{i} \in I$ and assume that $\theta_{\vec{i}} = \indep{\vec{a}}{\vec{b}}{\vec{c}}$. Let $\vec{\a}\vec{\b}\vec{\g}$ be such that $\vec{\a}\vec{\b}\in S^{\M}_{\vec{i}} $ and  $\vec{\a}\vec{\g}\in T^{\M}_{\vec{i}} $. Then there are $s,s' \in Y_{\vec{i}}$ such that $s(\vec{a}\vec{b})=\vec{\a}\vec{\b}$ and $s'(\vec{a}\vec{c})=\vec{\a}\vec{\g}$. By \eqref{indep} we can choose $s'' \in Y_{\vec{i}}$ such that $s''(\vec{a}\vec{b}\vec{c})=s(\vec{a}\vec{b})s'(\vec{c})$. Let us then choose the values for $\vec{z}_{\vec{i}}$ according to $s''$. Then all the values of $x_1, \ldots ,x_n$ agree with $s''$. Now,  
 since  $\M^* \models_{Y_{\vec{j}}} \varphi_{\vec{j}}$ for all  $\vec{j}$, and $s''\in Y_{\vec{j}}$ for $\vec{j} \leq \vec{i}$, it follows by Theorem \ref{flatness}
$$\M^* \models_{s''} \bigwedge_{\vec{j} \leq \vec{i}}\varphi_{\vec{j}}.$$
Now it suffices to show that
$$\M^* \models_{s''} Q^1 x'_1 \ldots Q^n x'_n (\varphi' \ja \chi).$$
For each $x'_i$ with $Q^i=\on$ we choose a value for it so that, for some $t \in Y$, the values of $x'_1, \ldots ,x'_i$ are $t(x_1), \ldots ,t(x_i)$. In particular, if the values of $x'_1, \ldots ,x'_{i-1}$ agree with $s''$, then we choose $x'_i$ according to $s''$ also. Let $s^*$ be an extension of $s''$ which is constructed according to these rules. Now using the fact that $\M^* \models_{t} \varphi$ for all $t\in Y$, and the way $s^*$ was chosen,  we get
$$\M^* \models_{s^*} \varphi' \ja \chi.$$
Hence $\M \models \Phi$. This concludes the \textit{only if-part}.

\textit{If-part}: Assume that $\M \models \Phi$. Then we can find interpretations $\vec{S}^{\M}$, $\vec{T}^{\M}$ and $\vec{f}^{\M}$ such that
$$\M^*\models Q^1 x_1 \ldots Q^n x_n \varphi \ja \Omega$$
when $\M^*:= (\M, \vec{S}^{\M},\vec{T}^{\M},\vec{f}^{\M})$. Consider the usual semantic game for first-order logic where player $\on$ plays the role of verifier and player $\forall$ plays the role of falsifier. Then there is a winning strategy for player $\on$ in the semantic game for $Q^1 x_1 \ldots Q^n x_n \varphi \ja \Omega$ over $\M^*$. Let $Y$ consist of assignments $s: \{x_1, \ldots ,x_n\} \rightarrow M$ corresponding to every possible play of $x_1, \ldots ,x_n$ where player $\on$ follows her winning strategy. Analogously, let $Y'$ consist of assignments  $s: \{x_1, \ldots ,x_n\} \rightarrow M$ that correspond to every possible play of $x'_1, \ldots ,x'_n$ where player $\on$ follows her winning strategy. Let $X:=Y\cup Y'$. We will show that
$$\M \models_X \theta.$$

We know that $\M^* \models_X \varphi$. Let us now define $X_{\vec{i}}$, for $\vec{i} \in S$, as follows. Recall that $\varphi_{\lambda}=\varphi$ where $\lambda$ is the empty sequence. We also let $X_{\lambda}:=X$.
\begin{itemize}
\item If $\M^* \models_{X_{\vec{i}}}\varphi_{\vec{i}}$ where $\varphi_{\vec{i}} = \varphi_{\vec{i}0}\ja \varphi_{\vec{i}1}$, then we let $X_{\vec{i}0} :=X_{\vec{i}} $ and $X_{\vec{i}1} :=X_{\vec{i}} $.
\item If $\M^* \models_{X_{\vec{i}}} \varphi_{\vec{i}}$ where $\varphi_{\vec{i}} = \varphi_{\vec{i}0}\tai \varphi_{\vec{i}1}$, then we let $X_{\vec{i}0} :=\{s \in X_{\vec{i}}\mid \M^* \models_s \varphi_{\vec{i}0}\}$ and $X_{\vec{i}1} :=\{s \in X_{\vec{i}}\mid \M^* \models_s \varphi_{\vec{i}1}\}$.
\end{itemize}
From the construction it follows that $\M^* \models_{X_{\vec{i}}} \varphi_{\vec{i}}$, for $\vec{i} \in S$, and that $X_{\vec{i}0}\cup X_{\vec{i}1} =X_{\vec{i}} $ if $\varphi_{\vec{i}}$ is a disjunction.
We will now show that for each atomic $\theta_{\vec{i}}$, $\M \models_{X_{\vec{i}}} \theta_{\vec{i}}$:
\begin{enumerate}
\item If $\theta_{\vec{i}}$ is a first-order atom, then the claim follows from $\theta_{\vec{i}}=\varphi_{\vec{i}}$.

\item If $\theta_{\vec{i}}$ is $\dep(\vec{z},y)$, then the claim follows from $\M^* \models_{X_{\vec{i}}} f_{\vec{i}}(\vec{z})=y$.

\item Assume that $\theta_{\vec{i}}$ is $\indep{\vec{a}}{\vec{b}}{\vec{c}}$. Then $\M^* \models_{X_{\vec{i}}} S_{\vec{i}}(\vec{a}\vec{b}) \ja T_{\vec{i}}(\vec{a}\vec{c})$. Let $s,s'\in X_{\vec{i}}$ be such that $s(\vec{a})=s'(\vec{a})$. We have to show that there is $s''\in X_{\vec{i}}$ such that $s''(\vec{a}\vec{b}\vec{c})=s(\vec{a})s(\vec{b})s'(\vec{c})$. Now $\M^* \models \Omega$, so consider a play in the semantic game where player $\forall$ chooses first the conjunct with index $\vec{i}$ from $\Omega$, and then chooses $s(\vec{a})s(\vec{b})s'(\vec{c})$ as values for $\vec{a}\vec{b}\vec{c}$. Since $s(\vec{a})s(\vec{b}) \in S_{\vec{i}}^{\M}$ and $s(\vec{a})s'(\vec{c}) \in T_{\vec{i}}^{\M}$, then
 player $\on$ plays according to her strategy and chooses values for $\vec{z}_{\vec{i}}$ so that
$$\M^* \models_{s''}  \bigwedge_{\vec{j} \leq \vec{i}}\varphi_{\vec{j}} \ja Q^1 x'_1 \ldots Q^n x'_n (\varphi' \ja \chi)$$
where $s''$ is the assignment agreeing with the chosen values for $\vec{a}\vec{b}\vec{c}$ and $\vec{z}_{\vec{i}}$. Now we let player $\forall$ play each $x'_i$ with $Q^i=\forall$ as $s''(x_i)$. Then because of $\chi$ (defined in \eqref{chi}) player $\on$ must also play each $x'_i$ with $Q^i = \on$ as $s''(x_i)$. Hence $s''$ corresponds to a play of $x'_1, \ldots ,x'_n$, and thus $s'' \in X$. 

Since $\M^* \models_{s''}  \bigwedge_{\vec{j} \leq \vec{i}}\varphi_{\vec{j}}$, it is a straightforward induction to show that $s'' \in X_{\vec{i}}$. This concludes the step $3.$
\end{enumerate}
Now using the previous, a straightforward backward induction shows that $\M \models_{X} \theta$. It then suffices to show that there are functions 
$$F_i: \{\emptyset\}[F_1/x_1]\ldots [F_{i-1}/x_{i-1}] \rightarrow \Po (M)\setminus \{\emptyset\}$$
such that $F_i(s)= M$ if $Q^i = \forall$, and that
$$X=\{\emptyset\}[F_1/x_1] \ldots [F_n/x_n].$$
We define these functions inductively so that $\{\emptyset\}[F_1/x_1]\ldots [F_i/x_i] = X \upharpoonright \{x_1, \ldots ,x_i\}$, for $1 \leq i \leq n$. Assume that we have defined $F_1, \ldots ,F_i$ successfully. We will define $F_{i+1}$ as wanted. Assume first that $Q^{i+1} = \on$. Then for $s \in \{\emptyset\}[F_1/x_1]\ldots [F_i/x_i]$, we let 
$$F_{i+1}(s) = \{t(x_{i+1}) \mid t \in X,t \upharpoonright \{x_1, \ldots ,x_i \}= s\}.$$ 
By the induction assumption $F_{i+1}(s)$ is non-empty, though it may not be singleton in case there are multiple plays where values of $x_1, \ldots ,x_i$ (or $x'_1, \ldots ,x'_i$) agree with $s$. We note that $$\{\emptyset\}[F_1/x_1]\ldots [F_{i+1}/x_{i+1}] = X \upharpoonright \{x_1, \ldots ,x_{i+1}\}.$$

Assume then that $Q^{i+1} = \forall$.  For $s \in \{\emptyset\}[F_1/x_1]\ldots [F_i/x_i]$, we let $F_{i+1}(s)= \Po (M)$ and note that 
$$X \upharpoonright \{x_1, \ldots ,x_{i+1}\} \subseteq \{\emptyset\}[F_1/x_1]\ldots [F_{i+1}/x_{i+1}].$$
For the other direction, assume that $s \in \{\emptyset\}[F_1/x_1]\ldots [F_{i}/x_{i}]$ and let $a \in M$. We show that $s(a/x) \in X \upharpoonright \{x_1, \ldots ,x_{i+1}\}$. By the induction assumption $s \in X \upharpoonright \{x_1, \ldots ,x_i\}$, and thus there is a play of $x_1, \ldots ,x_n$ (or $x'_1, \ldots ,x'_n$) that agrees with $s$ in the first $i$ variables. Let $s'$ be the assignment corresponding to this play. Now instead of choosing $s'(x_{i+1})$ (or $s'(x'_{i+1})$) at move $i+1$, player $\forall$ can choose $a$ for $x_{i+1}$ (or for $x'_{i+1}$).  Let $t$ be an assignment that corresponds to some play with these moves for the first $i+1$ variables. Then $t \in X$ and $t \upharpoonright \{x_1, \ldots ,x_{i+1}\} = s(a/x_{i+1})$. This concludes the proof, and thus also the \textit{only if-part}.

Note that in $\Phi$ each function or relation variable has an arity at most $k$. This concludes the proof.
\end{proof}

\begin{thm}\label{teoreema}
$\ESOfarity{k} = \ind{k}.$
\end{thm}
\begin{proof}
Follows from Propositions \ref{leq} and \ref{geq}.
\end{proof}

This gives us immediately a corollary regarding inclusion logic. Recall that $\inc{k}$ denotes the class of inclusion logic sentences in which inclusion atoms of width at most $k$ (i.e. atoms of the form $\vec{a}\subseteq \vec{b}$ where $|\vec{a}|=|\vec{b}|\leq k$) may appear.


\begin{thm}
Assume $k \geq 2$. Then $\inc{k}\leq \ESOfarity{k}$.
\end{thm}
\begin{proof}
Using item \ref{thm:inc_ind} of Theorem \ref{two}, we first translate inclusion logic sentences to independence logic, and then apply  Proposition \ref{geq}. It is easy to check that this translation takes us to $\ESOfarity{k'}$, where $k' = \max \{k,2\}$.
\end{proof}
There is no hope of proving the other direction, since, e.g.,  even cardinality cannot be expressed in  $\FO(\subseteq)$ \cite{gallhella13}, but it is expressible in $\ESOfarity{1}$. 
Next we will show that $\ESOfarity{k}  \leq \indNR{2k+2}$.
\begin{thm}
$\ESOfarity{k}  \leq \indNR{2k+1} \leq \ESOfarity{2k+1}.$
\end{thm}
\begin{proof}
For the first inequality, note that $\ESOfarity{k}=\ddep{k}$ by \cite{durand11}, and  $\ddep{k} \leq \indNR{2k+1}$ by Theorem \ref{thm:DLINDNR}. The second inequality follows from Theorem \ref{teoreema}.
\end{proof}

\subsection{$\forall$-hierarchies}
In this section, we will examine the fragments $\FO(\mathcal C)(k\forall )$. We will prove that, contrary to the case of the fragments $\mathcal D(k\forall)$ \cite{durand11}, the following holds: 
\begin{enumerate}
\item If $\{\bot, \subseteq\} \subseteq \mathcal C$ then the hierarchy collapses at level $1$:  $\FO(\mathcal C)= \FO(\mathcal C)(1\forall)$;
\item If $\bot \in \mathcal C$ then it collapses at level 2: $\FO(\mathcal C)= \FO(\mathcal C)(2\forall)$.
\end{enumerate}

We will use the following result from \cite{vaananen13}:
\begin{prop}\label{joukotrans}
Let $\phi$ be a $\indNRlogic$ sentence. Then $\phi$ is equivalent to an formula of the form $\forall \vec{x} \exists \vec{y} (\theta \wedge \chi)$, where $\theta$ is a conjunction of pure independence atoms and $\chi$ is first-order and quantifier-free. 
\end{prop}
Since, as we saw in the Preliminaries,  we can define inclusion atoms and conditional independence atoms in terms of pure independence atoms, it follows at once that any sentence of $\FO(\dep(\ldots), \bot_{\rm c}, \subseteq)$ is equivalent to some sentence of the above form.

Using this, we will prove that
\begin{thm}
$\FO(\bot, \subseteq)(1\forall) = \FO(\dep(\ldots), \bot_{\rm c}, \subseteq)$.
\end{thm}
\begin{proof}
Let $\phi \in \FO(\dep(\ldots), \bot_{\rm c}, \subseteq)$. We will show that there exists a $\phi'\in \FO(\bot, \subseteq)(1\forall)$ equivalent to it.  As we said, we can assume that $\phi$ is of the form
$\forall x_1 \ldots \forall x_m \on x_{m+1} \ldots \on x_{m+n} (\theta \wedge \chi),$
where $\theta$ is a conjunction of pure independence atoms and $\chi$ is first-order and quantifier-free. Let us then define $\phi'$ as
$$\forall x_1 \on x_2 \ldots \on x_m \on x_{m+1} \ldots \on x_{m+n} (\bigwedge_{2 \leq i \leq m}(x_1 \subseteq x_i \ja x_1\ldots x_{i-1} \bot x_i) \ja \theta \wedge \chi).$$
We claim that $\phi'$ is equivalent to $\phi$. Assume first that $\M \models \phi$. Then there are, for $m+1 \leq i \leq m+n$, functions
$$F_i: \{ \emptyset \}[M/x_1]\ldots [M/x_n][F_{m+1}/x_{m+1}]\ldots [F_{i-1}/x_{i-1}] \rightarrow \Po (M) \setminus \{\emptyset\}$$
such that $\M \models_X \theta \wedge \chi$ when $X:=[M/x_1]\ldots [M/x_n][F_{m+1}/x_{m+1}]\ldots [F_{m+n}/x_{m+n}]$. Let $F_i$, for $2 \leq i \leq m$, be the constant function mapping each assignment to $M$. Then
$$X=\{ \emptyset \}[M/x_1][F_2/x_2]\ldots [F_{m+n}/x_{m+n}].$$
Clearly $\M \models_X \bigwedge_{2 \leq i \leq m}(x_1 \subseteq x_i \ja x_1\ldots x_{i-1} \bot x_i)$, and hence $\M \models \phi'$.

For the other direction, assume that $\M \models \phi'$. Then there are, for $2 \leq i \leq m+n$, functions
$$F_i: \{ \emptyset \}[M/x_1][F_2/x_2]\ldots [F_{i-1}/x_{i-1}] \rightarrow \Po (M) \setminus \{\emptyset\}$$
such that $\M \models_X \bigwedge_{2 \leq i \leq m}(x_1 \subseteq x_i \ja x_1\ldots x_{i-1} \bot x_i)$ when $X:=\{ \emptyset \}[M/x_1][F_2/x_2]\ldots [F_{m+n}/x_{m+n}]$. Define, for $2 \leq i \leq m$, $X_i:= \{ \emptyset \}[M/x_1][F_2/x_2]\ldots [F_{i}/x_{i}]$ and $Y_i:=\{ \emptyset \}[M/x_1][M/x_2]\ldots [M/x_i]$. It suffices to show that $X_i=Y_i$ for $ 2 \leq i \leq m$. 

First let us prove the claim for $i=2$. Let $s \in Y_2$. It suffices to show that $s \in X_2$. By Proposition \ref{thm:loc}, $\M \models_{X_2} x_1 \subseteq x_2 \ja x_1 \bot x_2$. Let $s'\in X_2$ be such that $s'(x_1)=s(x_2)$. Since $\M \models_{X_2} x_1 \subseteq x_2$, we can find a $t\in X_2$ such that $t(x_2)=s'(x_1)$. Now let $t' \in X_2$ be such that $t'(x_1)=s(x_1)$. Because $\M \models_{X_2} x_1 \bot x_2$, we can find a $t'' \in X_2$ such that $t''(x_1)=t'(x_1)$ and $t''(x_2)=t(x_2)$. Then $t''=s$ which concludes the claim for $i=2$.

The induction step is proved analogously. This concludes the claim and the proof.
\end{proof}

Let us now prove our second claim.
\begin{thm}
$\FO(\bot)(2\forall) = \FO(\dep(\ldots), \bot_{\rm c}, \subseteq)$.
%
\end{thm}
\begin{proof}
Let $\phi \in \FO(\dep(\ldots), \bot_{\rm c}, \subseteq)$. Again, we can assume that $\phi$ is of the form  $\forall \vec{x} \psi$, where $\vec{x} = x_1 \ldots x_n$ and $\psi$ is of the form $\exists \vec y \theta$ for $\theta$ quantifier-free and in $\FO(\bot)$. Let now $p, q$ be two variables not occurring in $\phi$. We state that $\phi$ is equivalent to 
\[
	\phi^* = \forall p \forall q \exists \vec{x} \left( \left(p = q \rightarrow \bigwedge_{i=1}^n x_i = p\right) \wedge \bigwedge_{i=1}^{n-1} (x_1 \ldots x_i \bot x_{i+1}) \wedge \psi\right).
\]

Indeed, let $\M$ be a model and  $X = \{\emptyset\}[M/p][M/q]$, and let the tuple of (nondeterministic) choice functions $\vec{U}$ for $\vec{x}$ be such that 
\[
	\vec{U}(s) = \left\{
		\begin{array}{l l}
			(m,\ldots,m) & \mbox{ if } s(p) = s(q) = m;\\
			M^n & \mbox{ otherwise}
		\end{array}
	\right.
\]
and let $Y = X[\vec{U}/\vec{x}]$. It is obvious that $\M \models_Y \left(p = q \rightarrow \bigwedge_{i=1}^n x_i = p\right)$; and $\M \models_Y \psi$, because $Y(\vec{x}) = M^n$ and $p$, $q$ do not occur in $\psi$. Finally, it is also true that $Y$ satisfies all independence atoms $x_1 \ldots x_i \bot x_{i+1}$, since $Y(x_1 \ldots x_i x_{i+1}) = M^{i+1}$ (assuming that our model contains two distinct elements). Therefore $\M \models \phi^*$, as required.

Conversely, suppose that $\M \models \phi^*$: then there exists a $\vec{U}$ such that, for $Y = \{\emptyset\}[M/pq][\vec{U}/\vec{x}]$, $\M \models_Y \left(p = q \rightarrow \bigwedge_{i=1}^n x_i = p\right) \wedge \bigwedge_{i=1}^{n-1} (x_1 \ldots x_i \bot x_{i+1}) \wedge \psi$. We will show that $Y(x_1 \ldots x_n)$ is $M^n$, that is, that all possible tuples $m_1 \ldots m_n$ of elements of our models are possible values for $x_1 \ldots x_n$ in $Y$. 

First of all, let us observe that for all $m \in M$ there exists a $h^m \in Y$ such that $h^m(x_i) = m$ for all $i$. Indeed, we can find a $s^m \in X$ such that $s^m(p) = s^m(q) = m$ and then pick an arbitrary $h^m \in s^m[\vec{U}/\vec{x}] \subseteq Y$. Since $\M \models_Y p = q \rightarrow \bigwedge_i x_i = p$, we have at once that $h^m(x_i) = h^m(p) = m$, as required. 

Now we prove, by induction on $i = 1 \ldots n$, that there exists a $h_i \in Y$ such that $h_i(x_1 \ldots x_i) = m_1 \ldots m_i$. 
\begin{description}
\item[Base Case:] Let $h_1$ be $h^{m_1} \in Y$. Then $h^{m_1}(x_1) = m_1$, as required.
\item[Induction Case:] Suppose that $h_i(x_1 \ldots x_i) = m_1 \ldots m_i$, and consider $h^{m_{i+1}}$. As we saw, $h^{m_{i+1}} \in Y$ and $h^{m_{i+1}}(x_{i+1}) = m_{i+1}$. But $\M \models_Y x_1 \ldots x_i \bot x_{i+1}$; and therefore there exists a $h_{i+1} \in Y$ with $h_{i+1}(x_1 \ldots x_i) = h_i(x_1 \ldots x_i) = m_1 \ldots m_i$ and $h_{i+1}(x_{i+1}) = h^{m_{i+1}}(x_{i+1}) = m_{i+1}$. Hence, $h_{i+1}(x_1 \ldots x_{i+1}) = m_1 \ldots m_{i+1}$. 
\end{description}
In particular, this implies that $h_n(x_1 \ldots x_n) = m_1 \ldots m_n$; and since we started from an arbitrary choice of $m_1 \ldots m_n$, we can conclude that $Y(\vec{x}) = M^{|\vec{x}|}$. But then the restriction of $Y$ to $\vec{x}$ is precisely $\{\emptyset\}[M/\vec{x}]$; and since $\M \models_Y \psi$, by locality we have that $\M \models \forall \vec{x} \psi$, as required.
\end{proof}
%

\section*{Conclusion}
In this paper, we examined the expressive power of fragments of inclusion and independence logic obtained by restricting the arity of non first-order atoms or the number of universal quantifiers. For the first kind of restriction, we adapted and extended the hierarchy theorems of \cite{durand11} to this new setting; but for the second kind of restriction, we showed that the hierarchy collapses at a very low level if our logic contains at least pure independence atoms. 

A question which is still open is whether the fragments $\FO(\subseteq)(k\forall)$ of  inclusion logic give rise to an infinite expressivity hierarchy. Another issue that requires further investigation is to which degree our results can be adapted to the case of strict semantics. The exact nature of the relationship between strict and lax semantics is a matter which is of no small interest for the further development of the area, and a comparison of the properties of our fragments in these two settings might prove itself of great value. 

%
%
%






\bibliography{biblio}







\end{document}